\documentclass[a4paper]{amsart}
\usepackage{amscd}
\usepackage{amsmath}
\usepackage{mathrsfs}
\usepackage{amssymb}
\usepackage{amsthm}
\usepackage{bbm}
\usepackage{stmaryrd}

\usepackage[T1]{fontenc}

\makeatletter
\@namedef{subjclassname@2020}{%
  \textup{2020} Mathematics Subject Classification}
\makeatother

\newif\ifpdf
\ifx\pdfoutput\undefined
   \pdffalse        
\else
   \pdfoutput=1     
   \pdftrue
\fi

\ifpdf
   \usepackage[pdftex]{graphicx}
\else
   \usepackage{graphicx}
\fi

\numberwithin{equation}{section} \swapnumbers

\newtheorem{satz}{Satz}[section]

\newtheorem{theorem}[satz]{Theorem}
\newtheorem{proposition}[satz]{Proposition}

\newtheorem{assumption}[satz]{Assumption}

\newtheorem{definition}[satz]{Definition}

\newtheorem{remark}[satz]{Remark}

\newtheorem{example}[satz]{Example}

\newcommand{\bbr}{\mathbb{R}}

\newcommand{\bbn}{\mathbb{N}}
\newcommand{\bbp}{\mathbb{P}}

\newcommand{\calb}{\mathscr{B}}

\newcommand{\calf}{\mathscr{F}}

\newcommand{\calh}{\mathscr{H}}

\newcommand{\calz}{\mathscr{Z}}

\newcommand{\la}{\langle}
\newcommand{\ra}{\rangle}

\newcommand{\bbI}{\mathbbm{1}}

\begin{document}

\title[Mild solutions to stochastic partial differential equations]{An addendum to ``Mild solutions to semilinear stochastic partial differential equations with locally monotone coefficients''}
\author{Stefan Tappe}
\address{Albert Ludwig University of Freiburg, Department of Mathematical Stochastics, Ernst-Zermelo-Stra\ss{}e 1, D-79104 Freiburg, Germany}
\email{stefan.tappe@math.uni-freiburg.de}
\date{10 March, 2022}
\thanks{I am grateful to Michael R\"{o}ckner for fruitful discussions. I am also grateful to the referee for helpful comments. Moreover, I gratefully acknowledge financial support from the Deutsche Forschungsgemeinschaft (DFG, German Research Foundation) --- project number 444121509.}
\begin{abstract}
In this addendum we provide an existence and uniqueness result for mild solutions to semilinear stochastic partial differential equations driven by Wiener processes and Poisson random measures in the framework of the semigroup approach with locally monotone coefficients, where the semigroup is allowed to be pseudo-contractive. This improves an earlier paper of the author, where the equation was only driven by Wiener processes, and where the semigroup was only allowed to be a semigroup of contractions.
\end{abstract}

\maketitle

\section{Introduction}

This note is an addendum to an earlier paper by the author \cite{Tappe-mon}. In the aforementioned article, we have provided an existence and uniqueness result for mild solutions to semilinear stochastic partial differential equations (SPDEs) driven by Wiener processes in the framework of the semigroup approach (see, for example \cite{Da_Prato, Atma-book}) under local monotonicity and coercivity conditions on the coefficients; see \cite[Thm. 2.6]{Tappe-mon}. The semigroup was assumed to be a semigroup of contractions. The goal of this note is to improve this existence and uniqueness result for SPDEs driven by Wiener processes and Poisson random measures and pseudo-contractive semigroups; see Theorem \ref{thm-SPDE} below. We will also present consequences for such SPDEs with Lipschitz type coefficients, where the Lipschitz constants may be random; see Propositions \ref{prop-SPDE} and \ref{prop-SPDE-Wiener}, and also Example \ref{example-mult}, where the coefficients have a multiplicative structure.

The essential idea for the proof of the existence and uniqueness result from \cite{Tappe-mon} was to utilize the ``method of the moving frame'', which has originally been presented in \cite{SPDE}. This method allows to extend the semigroup to a group on a larger Hilbert space, which provides a link between mild solutions to SPDEs and strong solutions to infinite dimensional stochastic differential equations (SDEs) on the larger Hilbert space. Regarding the SDE on the larger Hilbert space, we have used an existence and uniqueness result for SPDEs driven by Wiener processes in the framework of the variational approach; see \cite[Thm. 5.1.3]{Liu-Roeckner}. In this paper, we provide a slight extension of \cite[Thm 1.2]{BLZ}, which is an existence and uniqueness result for SPDEs driven by Wiener processes and Poisson random measures in the framework of the variational approach; see Theorem \ref{thm-SDE-var}. We will use this result later on.

If the semigroup is a semigroup of contractions, then the group on the larger Hilbert space is unitary. By virtue of this property, in \cite{Tappe-mon} we were able to transfer all the required conditions to the SDE on the larger Hilbert space; in particular those, where the inner products of the Hilbert spaces are involved.

In the present situation, where the semigroup is pseudo-contractive, the group on the larger Hilbert space does not need to be unitary. As a consequence, we cannot immediately follow the arguments from \cite{Tappe-mon}. However, it turns out that the group on the larger Hilbert has certain properties, which are similar to those of a unitary group; see Proposition \ref{prop-diagram} below. With slightly more extensive calculations, these properties of the group also allow us to transfer all the required conditions to the SDE on the larger Hilbert space, which is the key for the proof of Theorem \ref{thm-SPDE}.

The remainder of this note is organized as follows. In Section \ref{sec-variational} we provide the mentioned existence and uniqueness result for SPDEs in the framework of the variational approach. Using this result, in Section \ref{sec-SDE} we show an existence and uniqueness result for infinite dimensional SDEs. Afterwards, in Section \ref{sec-SPDE} we provide the improved existence and uniqueness result for SPDEs in the framework of the semigroup approach, and present further consequences.

\section{Stochastic partial differential equations in the framework of the variational approach}\label{sec-variational}

In this section we provide the announced existence and uniqueness result for SPDEs in the framework of the variational approach. Let $T > 0$ be a finite time horizon, and let $(\Omega,\calf,(\calf_t)_{t \in [0,T]},\bbp)$ be a filtered probability space satisfying the usual conditions. Let $(V, H, V^*)$ be a Gelfand triplet. This means that $H$ is a separable Hilbert space, and that $V$ a reflexive Banach space such that $V \subset H$ continuously and densely. Then, by identification we have $V \subset H \subset V^*$, where $V^*$ denotes the dual space of $V$. For what follows, we will use the notation ${}_{V^*}\la v^*,v \ra_V := v^*(v)$ for $v^* \in V^*$ and $v \in V$. We refer, for example to \cite[Sec. 4.1]{Prevot-Roeckner} or \cite[Sec. 4.1]{Liu-Roeckner} for further details about Gelfand triplets. Let $U$ be a separable Hilbert space, and let $W = (W_t)_{t \in [0,T]}$ be a cylindrical Wiener process in $U$. We denote by $L_2(U,H)$ the space of all Hilbert-Schmidt operators from $U$ to $H$. Let $(Z,\calz)$ be a measurable space, and let $N$ be a Poisson random measure on $Z$ with compensator $\nu(dz) \otimes dt$ for some $\sigma$-finite measure $\nu$ on $(Z,\calz)$. The compensated Poisson random measure is denoted by $\tilde{N}(dt,dz) := N(dt,dz) - \nu(dz)dt$. Let $D \in \calz$ be such that $\nu(D) < \infty$. We denote by $\calb \calf$ be the $\sigma$-field of all progressively measurable sets on $[0,T] \times \Omega$, that is
\begin{align*}
\calb \calf = \{ A \subset [0,T] \times \Omega : A \cap ([0,t] \times \Omega) \in \calb([0,t]) \otimes \calf_t \text{ for all } t \in [0,T] \}.
\end{align*}
We consider the SPDE
\begin{align}\label{SPDE-var}
\left\{
\begin{array}{rcl}
dX_t & = & A(t,X_t) dt + B(t,X_t)dW_t + \int_{D^c} f(t,X_{t-},z) \tilde{N}(dt,dz)
\\ &  & + \int_D g(t,X_{t-},z) N(dt,dz) \medskip
\\ X_0 & = & x_0,
\end{array}
\right.
\end{align}
where $A : [0,T] \times \Omega \times V \to V^*$ and $B : [0,T] \times \Omega \times V \to L_2(U,H)$ are $\calb \calf \otimes \calb(V)$-measurable functions, and $f,g : [0,T] \times \Omega \times V \times Z \to H$ are $\calb \calf \otimes \calb(V) \otimes \calz$-measurable functions.

\begin{assumption}\label{ass-SPDE-var}
We assume there are constants $\alpha \in (1,\infty)$, $\beta \in \bbr_+$, $\theta \in (0,\infty)$, $\bar{C} \in \bbr_+$ and a nonnegative adapted process $F \in L^1([0,T] \times \Omega; dt \otimes \bbp)$ such that for all $u,v,\vartheta \in V$ and all $(t,\omega) \in [0,T] \times \Omega$ the following conditions are fulfilled:
\begin{enumerate}
\item[(H1)] (Hemicontinuity) The map $\lambda \mapsto {}_{V^*}\la A(t,\omega,u+\lambda v),\vartheta \ra_V$ is continuous on $\bbr$.

\item[(H2')] (Local monotonicity) We have
\begin{align*}
&2 {}_{V^*}\la A(t,\omega,u) - A(t,\omega,v), u-v \ra_V + \| B(t,\omega,u) - B(t,\omega,v) \|_{L_2(U,H)}^2
\\ &+ \int_Z \| f(t,\omega,u,z) - f(t,\omega,v,z) \|_H^2 \nu(dz) \leq (F(t,\omega) + \rho(v)) \| u-v \|_H^2,
\end{align*}
where $\rho : V \to [0,\infty)$ is a measurable, bounded on balls function.

\item[(H3)] (Coercivity) We have
\begin{align*}
2 {}_{V^*}\la A(t,\omega,v),v \ra_V + \| B(t,\omega,v) \|_{L_2(U,H)}^2 \leq \bar{C} \| v \|_H^2 - \theta \| v \|_V^{\alpha} + F(t,\omega).
\end{align*}
\item[(H4')] (Growth) We have
\begin{align*}
\| A(t,\omega,v) \|_{V^*}^{\frac{\alpha}{\alpha - 1}} \leq ( F(t,\omega) + \bar{C} \| v \|_V^{\alpha} ) (1 + \| v \|_H^{\beta}).
\end{align*}
\end{enumerate}
\end{assumption}

\begin{definition}
Given an $\calf_0$-measurable random variable $x_0 : \Omega \to H$, an $H$-valued c\`{a}dl\`{a}g adapted process $X = \{ X_t \}_{t \in [0,T]}$ is called a \emph{strong solution} to the SPDE (\ref{SPDE-var}) with $X_0 = x_0$ if for its $dt \otimes \bbp$-equivalence class $\hat{X}$ the following conditions are fulfilled:
\begin{enumerate}
\item We have $\hat{X} \in L^{\alpha}([0,T];V) \cap L^2([0,T];H)$.

\item We have $\bbp$-almost surely
\begin{equation}\label{eqn-SDE}
\begin{aligned}
X_t &= x_0 + \int_0^t A(s,\bar{X}_s) ds + \int_0^t B(s,\bar{X}_s) dW_s
\\ &\quad + \int_0^t \int_{D^c} f(s,\bar{X}_s,z) \tilde{N}(ds,dz) + \int_0^t \int_D g(s,\bar{X}_s,z) N(ds,dz), \quad t \in [0,T],
\end{aligned}
\end{equation}
where $\bar{X}$ is any $V$-valued progressively measurable $dt \otimes \bbp$-version of $\hat{X}$.
\end{enumerate}
\end{definition}

\begin{remark}
The stochastic integrals appearing in equation (\ref{eqn-SDE}) are specified as follows:
\begin{enumerate}
\item For a suitable $L_2(U,H)$-valued process $B$ we denote by
\begin{align*}
\int_0^t B_s dW_s, \quad t \in [0,T]
\end{align*}
the stochastic integral with respect to the cylindrical Wiener process $W$, as defined, for example, in \cite[Sec. 2.5]{Liu-Roeckner}.

\item For a suitable $H$-valued process $f$ we denote by
\begin{align*}
\int_0^t \int_{D^c} f(s,z) \tilde{N}(ds,dz), \quad t \in [0,T]
\end{align*}
the stochastic integral with respect to the compensated Poisson random measure $\tilde{N}$, as defined, for example, in \cite[Sec. 2.3]{SPDE}.

\item For a suitable $H$-valued process $g$ we denote by
\begin{align*}
\int_0^t \int_D g(s,z) N(ds,dz), \quad t \in [0,T]
\end{align*}
the stochastic integral with respect to the Poisson random measure $N$. Since $\nu(D) < \infty$, this integral can be defined pathwise for every $\omega \in \Omega$.
\end{enumerate}
\end{remark}

\begin{theorem}\label{thm-SDE-var}
Suppose that Assumption \ref{ass-SPDE-var} is satisfied for some $F \in L^{p/2}([0,T] \times \Omega; dt \otimes \bbp)$ with $p := \beta + 2$, and that there are constants $C \in \bbr_+$ and $\kappa \in [0,\frac{\theta}{2 \beta})$, where we agree on the condition $\kappa \in [0,\infty)$ in case $\beta = 0$, such that for all $(t,\omega,v) \in [0,T] \times \Omega \times V$ we have
\begin{align}\label{b-zusatzbedingung}
\| B(t,\omega,v) \|_{L_2(U,H)}^2 + \int_Z \| f(t,\omega,v,z) \|_H^2 \nu(dz) &\leq F(t,\omega) + C \| v \|_H^2 + \kappa \| v \|_V^{\alpha},
\\ \int_Z \| f(t,\omega,v,z) \|_H^p \nu(dz) &\leq F(t,\omega)^{p/2} + C \| v \|_H^p,
\\ \label{rho-zusatzbedingung} \rho(v) &\leq C(1 + \| v \|_V^{\alpha}) (1 + \| v \|_H^{\beta}).
\end{align}
Then for each $x_0 \in L^p(\Omega,\calf_0,\bbp;H)$ there exists a unique strong solution $X = \{ X_t \}_{t \in [0,T]}$ to the SPDE (\ref{SPDE-var}) with $X_0 = x_0$.
\end{theorem}

\begin{proof}
The proof is a consequence of \cite[Thm. 1.2]{BLZ}. Note that in \cite{BLZ} the local monotonicity condition (H2') is only stated with the process $F$ replaced by some constant $K$. However, the local monotonicity condition is only used once for the existence proof (see pages 304--306 in \cite{BLZ}) and once for the uniqueness proof (see pages 306--307 in \cite{BLZ}). The same arguments hold true with the constant $K$ replaced by the process $F$.
\end{proof}

\section{Infinite dimensional stochastic differential equations}\label{sec-SDE}

In this section we provide the existence and uniqueness result for infinite dimensional SDEs. The stochastic framework is as in Section \ref{sec-variational}. Let $\calh$ be a separable Hilbert space. We consider the $\calh$-valued SDE
\begin{align}\label{SDE}
\left\{
\begin{array}{rcl}
dY_t & = & a(t,Y_t) d t + b(t,Y_t) d W_t + \int_{D^c} c(t,Y_{t-},z) \tilde{N}(dt,dz) 
\\ && \int_{D} d(t,Y_{t-},z) N(dt,dz) \medskip
\\ Y_0 & = & y_0,
\end{array}
\right.
\end{align}
where $a : [0,T] \times \Omega \times \calh \to \calh$ and $b : [0,T] \times \Omega \times \calh \to L_2(U,\calh)$ are $\calb \calf \otimes \calb(\calh)$-measurable functions, and $c,d : [0,T] \times \Omega \times \calh \times Z \to \calh$ are $\calb \calf \otimes \calb(\calh) \otimes \calz$-measurable functions.

Given an $\calf_0$-measurable random variable $y_0 : \Omega \to \calh$, an $\calh$-valued c\`{a}dl\`{a}g adapted process $Y = \{ Y_t \}_{t \in [0,T]}$ is called a \emph{strong solution} to the SDE (\ref{SDE}) with $Y_0 = y_0$ if we have $\bbp$-almost surely
\begin{align*}
\int_0^T \bigg( \| a(s,Y_s) \|_{\calh} + \| b(s,Y_s) \|_{L_2(U,\calh)}^2 + \int_{D^c} \| c(s,Y_{s-},z) \|_{\calh}^2 \nu(d z) \bigg) d s < \infty
\end{align*}
as well as
\begin{align*}
Y_t &= y_0 + \int_0^t a(s,Y_s) d s + \int_0^t b(s,Y_s) d W_s + \int_0^t \int_{D^c} c(s,Y_{s-},z) \tilde{N}(d s,d z)
\\ &\quad + \int_0^t \int_{D} d(s,Y_{s-},z) N(d s,d z), \quad t \in [0,T].
\end{align*}

\begin{assumption}\label{ass-SDE}
We assume there are constants $\beta,\bar{C},\theta \in \bbr_+$ such that $\bar{C} \geq \theta > 0$ and a nonnegative adapted process $F \in L^1([0,T] \times \Omega; dt \otimes \bbp)$ such that for all $u,v,\vartheta \in \calh$ and all $(t,\omega) \in [0,T] \times \Omega$ the following conditions are fulfilled:
\begin{enumerate}
\item[(SD1)] (Hemicontinuity) The map $\lambda \mapsto \la a(t,\omega,u+\lambda v),\vartheta \ra_{\calh}$ is continuous on $\bbr$.

\item[(SD2')] (Local monotonicity) We have
\begin{align*}
&2\la a(t,\omega,u) - a(t,\omega,v), u-v \ra_{\calh} + \| b(t,\omega,u) - b(t,\omega,v) \|_{L_2(U,\calh)}^2
\\ &+ \int_Z \| c(t,\omega,u,z) - c(t,\omega,v,z) \|_{\calh}^2 \nu(dz) \leq (F(t,\omega) + \tau(\| v \|_{\calh})) \| u-v \|_{\calh}^2,
\end{align*}
where $\tau : \bbr_+ \to \bbr_+$ is a continuous, increasing function.

\item[(SD3)] (Coercivity) We have
\begin{align*}
2 \la a(t,\omega,v),v \ra_{\calh} + \| b(t,\omega,v) \|_{L_2(U,\calh)}^2 \leq (\bar{C}-\theta) \| v \|_{\calh}^2 + F(t,\omega).
\end{align*}
\item[(SD4')] (Growth) We have
\begin{align*}
\| a(t,\omega,v) \|_{\calh}^2 \leq ( F(t,\omega) + \bar{C} \| v \|_{\calh}^2 ) (1 + \| v \|_{\calh}^{\beta}).
\end{align*}
\end{enumerate}
\end{assumption}

\begin{theorem}\label{thm-SDE}
Suppose that Assumption \ref{ass-SDE} is satisfied for some $F \in L^{p/2}([0,T] \times \Omega; dt \otimes \bbp)$ with $p := \beta + 2$, and that there are constants $C \in \bbr_+$ and $\kappa \in [0,\frac{\theta}{2 \beta})$, where we agree on the condition $\kappa \in [0,\infty)$ in case $\beta = 0$, such that for all $(t,\omega,v) \in [0,T] \times \Omega \times \calh$ we have
\begin{align}\label{b-zusatz}
\| b(t,\omega,v) \|_{L_2(U,\calh)}^2 + \int_Z \| c(t,\omega,v,z) \|_{\calh}^2 \nu(dz) &\leq F(t,\omega) + (C + \kappa) \| v \|_{\calh}^2,
\\ \label{b-zusatz-2} \int_Z \| c(t,\omega,v,z) \|_{\calh}^p \nu(dz) &\leq F(t,\omega)^{p/2} + C \| v \|_{\calh}^p,
\\ \label{tau-zusatz} \tau(r) &\leq C(1 + r^2) (1 + r^{\beta}), \quad r \in \bbr_+.
\end{align}
Then for each $\calf_0$-measurable random variable $y_0 : \Omega \to \calh$ there exists a unique strong solution $Y$ to the SDE (\ref{SDE}) with $Y_0 = y_0$.
\end{theorem}

\begin{proof}
If $y_0 \in L^p(\Omega,\calf_0,\bbp;\calh)$, then the proof is a consequence of Theorem \ref{thm-SDE-var} with $V = H = \calh$, $\alpha = 2$ and $\rho : \calh \to \bbr_+$ given by $\rho(v) = \tau(\| v \|_{\calh})$ for $v \in \calh$. For an arbitrary $\calf_0$-measurable random variable $y_0 : \Omega \to \calh$, the assertion follows by considering the $\calf_0$-measurable partition $(\Omega_n)_{n \in \bbn}$ of $\Omega$ given by
\begin{align*}
\Omega_n := \{ \| y_0 \|_{\calh} \in [n-1,n) \}
\end{align*}
and the sequence $(y_0^n)_{n \in \bbn}$ of $\calf_0$-measurable random variables given by $y_0^n := y_0 \bbI_{\Omega_n}$ for each $n \in \bbn$.
\end{proof}

\section{Stochastic partial differential equations in the framework of the semigroup approach}\label{sec-SPDE}

In this section we provide the improved existence and uniqueness result for SPDEs in the framework of the semigroup approach. The stochastic framework is as in Section \ref{sec-variational}. Let $H$ be a separable Hilbert space, and let $A$ be the generator of a $C_0$-semigroup $(S_t)_{t \geq 0}$ on $H$. We consider the $H$-valued SPDE
\begin{align}\label{SPDE}
\left\{
\begin{array}{rcl}
dX_t & = & ( A + \alpha(t,X_t) ) d t + \sigma(t,X_t) d W_t + \int_{D^c} \gamma(t,X_{t-},z) \tilde{N}(dt,dz)
\\ && + \int_{D} \delta(t,X_{t-},z) N(dt,dz) \medskip
\\ X_0 & = & x_0,
\end{array}
\right.
\end{align}
where $\alpha : [0,T] \times \Omega \times H \to H$ and $\sigma : [0,T] \times \Omega \times H \to L_2(U,H)$ are $\calb \calf \otimes \calb(H)$-measurable functions, and $\gamma,\delta : [0,T] \times \Omega \times H \times Z \to H$ are $\calb \calf \otimes \calb(H) \otimes \calz$-measurable functions.

Given an $\calf_0$-measurable random variable $x_0 : \Omega \to \calh$, an $H$-valued c\`{a}dl\`{a}g adapted process $X = \{ X_t \}_{t \in [0,T]}$ is called a \emph{mild solution} to the SPDE (\ref{SPDE}) with $X_0 = x_0$ if we have $\bbp$-almost surely
\begin{align*}
\int_0^T \bigg( \| \alpha(s,X_s) \|_H + \| \sigma(s,X_s) \|_{L_2(U,H)}^2 + \int_E \| \gamma(s,X_{s-},z) \|_H^2 \nu(d z) \bigg) d s < \infty
\end{align*}
as well as
\begin{align*}
X_t &= S_t x_0 + \int_0^t S_{t-s} \alpha(s,X_s) d s + \int_0^t S_{t-s} \sigma(s,X_s) d W_s
\\ &\quad + \int_0^t \int_{D^c} S_{t-s} \gamma(s,X_{s-},z) \tilde{N}(d s,d z)
\\ &\quad + \int_0^t \int_D S_{t-s} \delta(s,X_{s-},z) N(d s,d z), \quad t \in [0,T].
\end{align*}

\begin{assumption}\label{ass-SPDE}
We assume there are constants $\beta, \bar{C}, \theta \in \bbr_+$ such that $\bar{C} \geq \theta > 0$ and a nonnegative adapted process $F \in L^1(\Omega \times [0,T]; dt \otimes \bbp)$ such that for all $x,y,\zeta \in H$ and all $(t,\omega) \in [0,T] \times \Omega$ the following conditions are fulfilled:
\begin{enumerate}
\item[(SP1)] (Hemicontinuity) The map $\lambda \mapsto \la \alpha(t,\omega,x+\lambda y),\zeta \ra_H$ is continuous on $\bbr$.

\item[(SP2')] (Local monotonicity) We have
\begin{align*}
&2\la \alpha(t,\omega,x) - \alpha(t,\omega,y), x-y \ra_H + \| \sigma(t,\omega,x) - \sigma(t,\omega,y) \|_{L_2(U,H)}^2
\\ &+ \int_Z \| \gamma(t,\omega,x,z) - \gamma(t,\omega,y,z) \|_H^2 \nu(dz) \leq (F(t,\omega) + \tau(\| y \|_H)) \| x-y \|_H^2,
\end{align*}
where $\tau : \bbr_+ \to \bbr_+$ is a continuous, increasing function.

\item[(SP3)] (Coercivity) We have
\begin{align*}
2 \la \alpha(t,\omega,y),y \ra_H + \| \sigma(t,\omega,y) \|_{L_2(U,H)}^2 \leq (\bar{C}-\theta) \| y \|_H^2 + F(t,\omega).
\end{align*}
\item[(SP4')] (Growth) We have
\begin{align*}
\| \alpha(t,\omega,y) \|_H^2 \leq ( F(t,\omega) + \bar{C} \| y \|_H^2 ) (1 + \| y \|_H^{\beta}).
\end{align*}
\end{enumerate}
\end{assumption}

\begin{theorem}\label{thm-SPDE}
We assume that the semigroup $(S_t)_{t \geq 0}$ is pseudo-contractive; that is, there is a constant $\eta \geq 0$ such that
\begin{align}\label{pseudo-contr}
\| S_t \| \leq e^{\eta t} \quad \text{for all $t \geq 0$.}
\end{align}
Furthermore, we suppose that Assumption \ref{ass-SPDE} is satisfied for some $F \in L^{p/2}([0,T] \times \Omega; dt \otimes \bbp)$ with $p := \beta + 2$, and that there are constants $C \in \bbr_+$ and $\kappa \in [0,\frac{\theta}{2 \beta})$, where we agree on the condition $\kappa \in [0,\infty)$ in case $\beta = 0$, such that for all $(t,\omega,y) \in [0,T] \times \Omega \times H$ we have
\begin{align}\label{sigma-zusatzbedingung}
\| \sigma(t,\omega,y) \|_{L_2(U,H)}^2 + \int_Z \| \gamma(t,\omega,y,z) \|_H^2 \nu(dz) &\leq F(t,\omega) + (C + \kappa) \| y \|_H^2,
\\ \label{gamma-zusatzbedingung} \int_Z \| \gamma(t,\omega,y,z) \|_H^p \nu(dz) &\leq F(t,\omega)^{p/2} + C \| y \|_H^p,
\\ \label{tau-zusatzbedingung} \tau(r) &\leq C(1 + r^2) (1 + r^{\beta}), \quad r \in \bbr_+.
\end{align}
Then for each $\calf_0$-measurable random variable $x_0 : \Omega \to H$ there exists a unique mild solution $X$ to the SPDE (\ref{SPDE}) with $X_0 = x_0$.
\end{theorem}

\begin{remark}
As already mentioned, Theorem \ref{thm-SPDE} generalizes \cite[Thm. 2.6]{Tappe-mon} in two aspects:
\begin{itemize}
\item We have an additional Poisson random measure $N$.

\item The semigroup $(S_t)_{t \geq 0}$ does not need to be a semigroup of contractions. Now, it is allowed to be pseudo-contractive.
\end{itemize}
Note that Theorem \ref{thm-SPDE} also generalizes \cite[Thm. 3.3]{Salavati}, because in our conditions we consider a random process $F$ rather than a constant.
\end{remark}

For the proof of Theorem \ref{thm-SPDE} we prepare a series of auxiliary results.

\begin{proposition}\label{prop-diagram}
There exist another separable Hilbert space $\calh$, a $C_0$-group $(U_t)_{t \in \mathbb{R}}$ on $\calh$ and an isometric embedding $\ell \in L(H,\calh)$ such that the following statements are true:
\begin{enumerate}
\item We have 
\begin{align}\label{diagram-commutes}
\pi U_t \ell = S_t \quad \text{for all $t \in \mathbb{R}_+$,}
\end{align}
where $\pi := \ell^*$ is the orthogonal projection from $\calh$ into $H$.

\item We have
\begin{align}\label{U-eqn-1}
U_{-t} = e^{-2 \eta t} U_t^* \quad \text{for all $t \in \bbr$.}
\end{align}

\item For each $t \in \bbr$ we have
\begin{align}\label{U-eqn-2}
\| U_t v \|_{\calh} = e^{\eta t} \| v \|_{\calh}, \quad v \in \calh.
\end{align}
\item In particular, we have $\| U_t \| = e^{\eta t}$ for each $t \in \bbr$.

\end{enumerate}
\end{proposition}

\begin{proof}
There exist another separable Hilbert space $\calh$ and a unitary group $(V_t)_{t \in \bbr}$ on $\calh$ such that the $C_0$-group $(U_t)_{t \in \bbr}$ given by
\begin{align}\label{U-in-proof}
U_t := e^{\eta t} V_t \quad \text{for each $t \in \bbr$} 
\end{align}
satisfies (\ref{diagram-commutes}). This follows from \cite[Prop. 8.7]{SPDE} and its proof. For convenience of the reader, let us briefly provide the details. The $C_0$-semigroup $(T_t)_{t \geq 0}$ defined as $T_t := e^{-\eta t} S_t$ for each $t \in \bbr_+$, where the constant $\eta \geq 0$ stems from (\ref{pseudo-contr}), is a semigroup of contractions. By the the Sz\H{o}kefalvi-Nagy theorem on unitary dilations (see e.g. \cite[Thm. I.8.1]{Nagy}, or \cite[Sec. 7.2]{Davies}), there exist another separable Hilbert space $\calh$, a unitary $C_0$-group $(V_t)_{t \in \mathbb{R}}$ on $\calh$ and an isometric embedding $\ell \in L(H,\calh)$ such that
\begin{align*}
\pi V_t \ell = T_t \quad \text{for all $t \in \mathbb{R}_+$,}
\end{align*}
where $\pi := \ell^*$ is the orthogonal projection from $\calh$ into $H$. Thus, defining the $C_0$-group $(U_t)_{t \in \bbr}$ by (\ref{U-in-proof}), we obtain
\begin{align*}
\pi U_t \ell = e^{\eta t} \pi V_t \ell = e^{\eta t} T_t = S_t \quad \text{for all $t \in \mathbb{R}_+$,}
\end{align*}
showing (\ref{diagram-commutes}). Now, let $t \in \bbr$ be arbitrary. Then we obtain
\begin{align*}
U_{-t} = e^{-\eta t} V_{-t} = e^{-\eta t} V_t^* = e^{-2 \eta t} U_t^*
\end{align*}
as well as
\begin{align*}
\| U_t v \|_{\calh} = e^{\eta t} \| V_t v \|_{\calh} = e^{\eta t} \| v \|_{\calh} \quad \text{for each $v \in \calh$,}
\end{align*}
completing the proof.
\end{proof}

\begin{remark}\label{rem-unitary}
Consider the particular situation $\eta = 0$ in (\ref{pseudo-contr}); that is $(S_t)_{t \geq 0}$ is a semigroup of contractions. Then by Proposition \ref{prop-diagram} the group $(U_t)_{t \in \bbr}$ is unitary, which is in accordance with \cite[Thm. 2.1]{Tappe-mon}.
\end{remark}

Now, we consider the $\calh$-valued SDE (\ref{SDE}) with coefficients $a : [0,T] \times \Omega \times \calh \to \calh$, $b : [0,T] \times \Omega \times \calh \to L_2(U,\calh)$ and  $c,d : [0,T] \times \Omega \times \calh \times Z \to \calh$ given by
\begin{align*}
a(t,\omega,y) &:= U_{-t} \ell \alpha(t,\omega,\pi U_t y),
\\ b(t,\omega,y) &:= U_{-t} \ell \sigma(t,\omega,\pi U_t y),
\\ c(t,\omega,y,z) &:= U_{-t} \ell \gamma(t,\omega,\pi U_t y,z),
\\ d(t,\omega,y,z) &:= U_{-t} \ell \delta(t,\omega,\pi U_t y,z).
\end{align*}
Note that $a$ and $b$ are $\calb \calf \otimes \calb(\calh)$-measurable, and that $c$ and $d$ are $\calb \calf \otimes \calb(\calh) \otimes \calz$-measurable.

\begin{proposition}\label{prop-X-Y}
Suppose that for each $\calf_0$-measurable random variable $y_0 : \Omega \to \calh$ there exists a unique strong solution $Y$ to the SDE (\ref{SDE}) with $Y_0 = y_0$. Then for each $\calf_0$-measurable random variable $x_0 : \Omega \to H$ there exists a unique mild solution $X$ to the SPDE (\ref{SPDE}) with $X_0 = x_0$.
\end{proposition}

\begin{proof}
Corollaries 3.9 and 3.11 from \cite{Tappe-YW} also hold true in the present situation with an additional Poisson random measure, with literally the same proofs. Combining these two results completes the proof.
\end{proof}

Now, we are ready to provide the proof of Theorem \ref{thm-SPDE}.

\begin{proof}[Proof of Theorem \ref{thm-SPDE}]
Using Proposition \ref{prop-diagram}, we will check that Assumption \ref{ass-SDE} is fulfilled. Let $u,v,\vartheta \in \calh$ and $(t,\omega) \in [0,T] \times \Omega$ be arbitrary.
\begin{enumerate}
\item[(SD1)] (Hemicontinuity) The map
\begin{align*}
\lambda \mapsto \la a(t,\omega,u+\lambda v),\vartheta \ra_{\calh} &= \la U_{-t} \ell \alpha(t,\omega,\pi U_t(u+\lambda v)),\vartheta \ra_{\calh}
\\ &= e^{-2 \eta t} \la U_t^* \ell \alpha(t,\omega,\pi U_t(u+\lambda v)),\vartheta \ra_{\calh}
\\ &= e^{-2 \eta t} \la \alpha(t,\omega,\pi U_t u + \lambda \pi U_t v), \pi U_t \vartheta \ra_H
\end{align*}
is continuous on $\bbr$. 

\item[(SD2')] (Local monotonicity) We define the continuous, increasing function $\tilde{\tau} : \bbr_+ \to \bbr_+$ as
\begin{align*}
\tilde{\tau}(r) := \tau(K r), \quad r \in \bbr_+,
\end{align*}
where the constant $K \geq 1$ is given by $K := e^{\eta T}$. Then by (\ref{tau-zusatzbedingung}) we have
\begin{align}\label{tau-tilde}
\tilde{\tau}(r) \leq \tilde{C}\left(1 + r^2\right) \left(1 + r^{\beta}\right), \quad r \in \bbr_+,
\end{align}
where $\tilde{C} := C K^{2 + \beta}$. Furthermore, we have
\begin{align*}
&2 \la a(t,\omega,u) - a(t,\omega,v), u-v \ra_{\calh} + \| b(t,\omega,u) - b(t,\omega,v) \|_{L_2(U,\calh)}^2
\\ &+ \int_Z \| c(t,\omega,u,z) - c(t,\omega,v,z) \|_{\calh}^2 \nu(dz)
\\ &=2 \la U_{-t} \ell \alpha(t,\omega,\pi U_t u) - U_{-t} \ell \alpha(t,\omega,\pi U_t v), u-v \ra_{\calh}
\\ &\quad + \| U_{-t} \ell \sigma(t, \omega, \pi U_t u) - U_{-t} \ell \sigma(t,\omega,\pi U_t v) \|_{L_2(U,\calh)}^2
\\ &\quad + \int_Z \| U_{-t} \ell \gamma(t,\omega,\pi U_t u,z) - U_{-t} \ell \gamma(t,\omega,\pi U_t v,z) \|_{\calh}^2 \nu(dz)
\\ &=2 e^{-2 \eta t} \la U_t^* \ell \alpha(t,\omega,\pi U_t u) - U_t^* \ell \alpha(t,\omega,\pi U_t v), u-v \ra_{\calh}
\\ &\quad + e^{-2 \eta t} \| \ell \sigma(t, \omega, \pi U_t u) - \ell \sigma(t,\omega,\pi U_t v) \|_{L_2(U,\calh)}^2
\\ &\quad + e^{-2 \eta t} \int_Z \| \ell \gamma(t,\omega,\pi U_t u,z) - \ell \gamma(t,\omega,\pi U_t v,z) \|_{\calh}^2 \nu(dz)
\\ &= e^{-2 \eta t} \bigg( 2 \la \alpha(t,\omega,\pi U_t u) - \alpha(t,\omega,\pi U_t v), \pi U_t(u-v) \ra_H
\\ &\qquad\qquad\,\, + \| \sigma(t,\omega,\pi U_t u) - \sigma(t,\omega,\pi U_t v) \|_{L_2(U,H)}^2 
\\ &\qquad\qquad\,\, + \int_Z \| \gamma(t,\omega,\pi U_t u,z) - \gamma(t,\omega,\pi U_t v,z) \|_H^2 \nu(dz) \bigg)
\\ &\leq e^{-2 \eta t} (F(t,\omega) + \tau(\| \pi U_t v \|_H) \| \pi U_t u - \pi U_t v \|_H^2
\\ &\leq (F(t,\omega) + \tau(e^{\eta T} \| v \|_{\calh}) \| u - v \|_{\calh}^2 = (F(t,\omega) + \tilde{\tau}(\| v \|_{\calh}) \| u - v \|_{\calh}^2.
\end{align*}

\item[(SD3)] (Coercivity) We have
\begin{align*}
&2 \la a(t,\omega,v),v \ra_{\calh} + \| b(t,\omega,v) \|_{L_2(U,\calh)}^2
\\ &= 2 \la U_{-t} \ell \alpha(t,\omega,\pi U_t v), v \ra_{\calh} + \| U_{-t} \ell \sigma(t,\omega,\pi U_t v) \|_{L_2(U,\calh)}^2
\\ &= 2 e^{-2 \eta t} \la U_t^* \ell \alpha(t,\omega,\pi U_t v), v \ra_{\calh} + e^{-2 \eta t} \| \ell \sigma(t,\omega,\pi U_t v) \|_{L_2(U,\calh)}^2
\\ &= e^{-2 \eta t} \big( 2 \la \alpha(t,\omega,\pi U_t v), \pi U_t v \ra_H + \| \sigma(t,\omega,\pi U_t v) \|_{L_2(U,H)}^2 \big)
\\ &\leq e^{-2 \eta t} \big( (\bar{C} - \theta) \| \pi U_t v \|_H^2 + F(t,\omega) \big)
\\ &\leq (\bar{C} - \theta) \| v \|_{\calh}^2 + F(t,\omega).
\end{align*}

\item[(SD4')] (Growth) We have
\begin{align*}
\| a(t,\omega,v) \|_{\calh}^2 &= \| U_{-t} \ell \alpha(t,\omega,\pi U_t v) \|_{\calh}^2 = e^{-2 \eta t} \| \alpha(t,\omega,\pi U_t v) \|_H^2
\\ &\leq e^{-2 \eta t} ( F(t,\omega) + \bar{C} \| \pi U_t v \|_H^2 ) ( 1 + \| \pi U_t v \|_H^{\beta} )
\\ &\leq ( F(t,\omega) + \bar{C} \| v \|_{\calh}^2 ) ( 1 + (e^{\eta t} \| v \|_{\calh})^{\beta} )
\\ &\leq K^{\beta} ( F(t,\omega) + \bar{C} \| v \|_{\calh}^2 ) ( 1 + \| v \|_{\calh}^{\beta} ),
\end{align*}
where we recall that $K = e^{\eta T}$.
\end{enumerate}
Furthermore, by (\ref{sigma-zusatzbedingung}) we have
\begin{align*}
&\| b(t,\omega,v) \|_{L_2(U,\calh)}^2 + \int_Z \| c(t,\omega,v,z) \|_{\calh}^2 \nu(dz) 
\\ &= \| U_{-t} \ell \sigma(t,\omega,\pi U_t v) \|_{L_2(U,\calh)}^2 + \int_Z \| U_{-t} \ell \gamma(t,\omega,\pi U_t v,z) \|_{\calh}^2 \nu(dz)
\\ &= e^{-2 \eta t} \bigg( \| \sigma(t,\omega,\pi U_t v) \|_{L_2(U,H)}^2 + \int_Z \| \gamma(t,\omega,\pi U_t v,z) \|_{H}^2 \nu(dz) \bigg)
\\ &\leq e^{-2 \eta t} \left(F(t,\omega) + (C + \kappa) \| \pi U_t v \|_H^2\right) \leq F(t,\omega) + (C + \kappa) \| v \|_{\calh}^2,
\end{align*}
and by (\ref{gamma-zusatzbedingung}) we have
\begin{align*}
&\int_Z \| c(t,\omega,v,z) \|_{\calh}^p \nu(dz) = \int_Z \| U_{-t} \ell \gamma(t,\omega,\pi U_t v,z) \|_{\calh}^p \nu(dz)
\\ &= e^{-p \eta t} \int_Z \| \gamma(t,\omega,\pi U_t v,z) \|_{H}^p \nu(dz)
\\ &\leq e^{-p \eta t} \left(F(t,\omega)^{p/2} + C \| \pi U_t v \|_H^p\right) \leq F(t,\omega)^{p/2} + C \| v \|_{\calh}^p,
\end{align*}
which shows (\ref{b-zusatz}) and (\ref{b-zusatz-2}). Moreover, by (\ref{tau-tilde}), condition (\ref{tau-zusatz}) is satisfied with $\tau$ replaced by $\tilde{\tau}$ and $C$ replaced by $\tilde{C}$. Consequently, all assumptions from Theorem \ref{thm-SDE} are fulfilled. Together with Proposition \ref{prop-X-Y}, the proof is completed.
\end{proof}

\begin{remark}
The results about the Markov property of solutions (see in particular \cite[Thm. 3.2]{Tappe-mon}) remain true in the present more general setting with literally the same proofs as in \cite[Sec. 3]{Tappe-mon}; merely $U_s^*$ should be replaced by $U_{-s}$, where it appears.
\end{remark}

We have the following consequence for SPDEs of the type (\ref{SPDE}) with Lipschitz type coefficients, where the Lipschitz constants may be random.

\begin{proposition}\label{prop-SPDE}
We assume that the semigroup $(S_t)_{t \geq 0}$ is pseudo-contractive. Furthermore, suppose there exist a nonnegative adapted process $f \in L^1([0,T] \times \Omega; dt \otimes \bbp)$, a constant $K \in \bbr_+$ and a continuous, increasing function $\chi : \bbr_+ \to \bbr_+$ satisfying
\begin{align}\label{chi-growth}
\chi(r) &\leq K(1 + r^2), \quad r \in \bbr_+
\end{align}
such that for all $x,y \in H$ and all $(t,\omega) \in [0,T] \times \Omega$ we have
\begin{align}\label{Lip-alpha}
\| \alpha(t,\omega,x) - \alpha(t,\omega,y) \|_H &\leq (f(t,\omega) + \chi(\| y \|_H)) \| x-y \|_H,
\\ \label{Lip-sigma} \| \sigma(t,\omega,x) - \sigma(t,\omega,y) \|_{L_2(U,H)}^2 &\leq (f(t,\omega) + \chi(\| y \|_H)) \| x-y \|_H^2,
\\ \label{Lip-gamma} \int_Z \| \gamma(t,\omega,x,z) - \gamma(t,\omega,y,z) \|_H^2 \nu(dz) &\leq (f(t,\omega) + \chi(\| y \|_H)) \| x-y \|_H^2,
\\ \label{LG-alpha} \| \alpha(t,\omega,y) \|_H &\leq \frac{f(t,\omega)}{1 + \| y \|_H}, 
\\ \label{LG-sigma} \| \sigma(t,\omega,y) \|_{L_2(U,H)}^2 &\leq K ( f(t,\omega) + \| y \|_H^2 ),
\\ \label{LG-gamma} \int_Z \| \gamma(t,\omega,y,z) \|_H^2 \nu(dz) &\leq K ( f(t,\omega) + \| y \|_H^2 ).
\end{align}
Then for each $\calf_0$-measurable random variable $x_0 : \Omega \to H$ there exists a unique mild solution $X$ to the SPDE (\ref{SPDE}) with $X_0 = x_0$.
\end{proposition}

\begin{proof}
We set $\beta := 0$, $\bar{C} := 6 (K \vee 1)$, $\theta := 3 (K \vee 1)$ and $F := 4 (K \vee 1) f$. We will check that Assumption \ref{ass-SPDE} is satisfied. Let $x,y \in H$ and $(\omega,t) \in [0,T] \times \Omega$ be arbitrary. By (\ref{LG-alpha}) we have
\begin{align}\label{alpha-1}
\| \alpha(t,\omega,y) \|_H \leq f(t,\omega),
\\ \label{alpha-2} \| \alpha(t,\omega,y) \|_H \| y \|_H \leq f(t,\omega).
\end{align}
\begin{enumerate}
\item[(SP1)] By (\ref{Lip-alpha}) the map $\xi \mapsto \alpha(t,\omega,\xi)$ is continuous, which proves the hemicontinuity.

\item[(SP2')] By (\ref{Lip-alpha})--(\ref{Lip-gamma}) we have
\begin{align*}
&2\la \alpha(t,\omega,x) - \alpha(t,\omega,y), x-y \ra_H + \| \sigma(t,\omega,x) - \sigma(t,\omega,y) \|_{L_2(U,H)}^2
\\ &\quad + \int_Z \| \gamma(t,\omega,x,z) - \gamma(t,\omega,y,z) \|_H^2 \nu(dz)
\\ &\leq 2 \| \alpha(t,\omega,x) - \alpha(t,\omega,y) \|_H \| x-y \|_H + \| \sigma(t,\omega,x) - \sigma(t,\omega,y) \|_{L_2(U,H)}^2
\\ &\quad + \int_Z \| \gamma(t,\omega,x,z) - \gamma(t,\omega,y,z) \|_H^2 \nu(dz)
\\ &\leq 4 (f(t,\omega) + \chi(\| y \|_H)) \| x-y \|_H^2
\\ &\leq (F(t,\omega) + \tau(\| y \|_H)) \| x-y \|_H^2,
\end{align*}
where the continuous, increasing function $\tau : \bbr_+ \to \bbr_+$ is given by $\tau := 4 \chi$.

\item[(SP3)] By (\ref{alpha-2}) and (\ref{LG-sigma}) we have
\begin{align*}
2 \la \alpha(t,\omega,y),y \ra_H + \| \sigma(t,\omega,y) \|_{L_2(U,H)}^2 &\leq 2 \| \alpha(t,\omega,y) \|_H \| y \|_H + \| \sigma(t,\omega,y) \|_{L_2(U,H)}^2
\\ &\leq 2 f(t,\omega) + K ( f(t,\omega) + \| y \|_H^2 )
\\ &\leq 3 (K \vee 1) ( f(t,\omega) + \| y \|_H^2 )
\\ &= (\bar{C}-\theta) \| y \|_H^2 + F(t,\omega).
\end{align*}
\item[(SP4')] If $\| y \|_H \leq 1$, then by (\ref{alpha-1}) we have
\begin{align*}
\| \alpha(t,\omega,y) \|_H \leq f(t,\omega) \leq F(t,\omega) + \bar{C} \| y \|_H^2,
\end{align*}
and if $\| y \|_H > 1$, then by (\ref{alpha-2}) we have
\begin{align*}
\| \alpha(t,\omega,y) \|_H \leq \| \alpha(t,y,\omega) \|_H \| y \|_H \leq f(t,\omega) \leq F(t,\omega) + \bar{C} \| y \|_H^2.
\end{align*}
\end{enumerate}
Now, recall that $\beta = 0$, which implies $p=2$. We set $C := 2K$ and $\kappa := 0$. Then by (\ref{LG-sigma}) and (\ref{LG-gamma}) we have
\begin{align*}
\| \sigma(t,y,\omega) \|_{L_2(U,H)}^2 + \int_Z \| \gamma(t,\omega,y,z) \|_H^2 \nu(dz) &\leq 2K ( f(t,\omega) + \| y \|_H^2 )
\\ &\leq F(t,\omega) + C \| y \|_H^2,
\end{align*}
showing that conditions (\ref{sigma-zusatzbedingung}) and (\ref{gamma-zusatzbedingung}) are satisfied, where we recall hat $p=2$ and $\kappa = 0$. Moreover, by (\ref{chi-growth}) we have
\begin{align*}
\tau(r) = 4 \chi(r) \leq 4K(1+r^2) = 2K(1+r^2)(1+r^{\beta}) = C(1+r^2)(1+r^{\beta}), \quad r \in \bbr_+,
\end{align*}
where we recall that $\beta = 0$. Hence, condition (\ref{tau-zusatzbedingung}) is satisfied as well. Consequently, applying Theorem \ref{thm-SPDE} concludes the proof.
\end{proof}

In the pure diffusion case
\begin{align}\label{SPDE-Wiener}
\left\{
\begin{array}{rcl}
dX_t & = & ( A + \alpha(t,X_t) ) d t + \sigma(t,X_t) d W_t \medskip
\\ X_0 & = & x_0
\end{array}
\right.
\end{align}
the previous result can be generalized as follows.

\begin{proposition}\label{prop-SPDE-Wiener}
We assume that the semigroup $(S_t)_{t \geq 0}$ is pseudo-contractive. Furthermore, suppose there exist $\beta \in \bbr_+$, a nonnegative adapted process $f \in L^{p/2}([0,T] \times \Omega; dt \otimes \bbp)$, where $p := \beta + 2$, a constant $K \in \bbr_+$ and a continuous, increasing function $\chi : \bbr_+ \to \bbr_+$ satisfying
\begin{align*}
\chi(r) &\leq K(1 + r^2)(1 + r^{\beta}), \quad r \in \bbr_+
\end{align*}
such that for all $x,y \in H$ and all $(t,\omega) \in [0,T] \times \Omega$ we have (\ref{Lip-alpha}), (\ref{Lip-sigma}) and (\ref{LG-alpha}), (\ref{LG-sigma}). Then for each $\calf_0$-measurable random variable $x_0 : \Omega \to H$ there exists a unique mild solution $X$ to the SPDE (\ref{SPDE-Wiener}) with $X_0 = x_0$.
\end{proposition}

\begin{proof}
The proof is similar to that of Proposition \ref{prop-SPDE}, and therefore omitted.
\end{proof}

In the following example the coefficients of the SPDE (\ref{SPDE}) have a multiplicative structure.

\begin{example}\label{example-mult}
Let $\bar{\alpha} : H \to H$, $\bar{\sigma} : H \to L_2(U,H)$ and $\bar{\gamma} : H \times Z \to H$ be measurable functions. We assume that they are Lipschitz continuous; that is, there is a constant $L \in \bbr_+$ such that for all $x,y \in H$ we have
\begin{align*}
\| \bar{\alpha}(x) - \bar{\alpha}(y) \|_H &\leq L \| x-y \|_H,
\\ \| \bar{\sigma}(x) - \bar{\sigma}(y) \|_{L_2(U,H)} &\leq L \| x-y \|_H,
\\ \bigg( \int_E \| \bar{\gamma}(x,z) - \bar{\gamma}(y,z) \|_H^2 \nu(dz) \bigg)^{1/2} &\leq L \| x-y \|_H.
\end{align*}
Furthermore, we assume there is a constant $C \in \bbr_+$ such that for all $y \in H$ we have
\begin{align*}
\| \bar{\alpha}(y) \|_H &\leq \frac{C}{1 + \| y \|_H},
\\ \| \bar{\sigma}(y) \|_{L_2(U,H)} &\leq C,
\\ \bigg( \int_E \| \bar{\gamma}(y,z) \|_H^2 \nu(dz) \bigg)^{1/2} &\leq C.
\end{align*}
Moreover, let $f_{\alpha} \in L^1([0,T] \times \Omega; dt \otimes \bbp)$ and $f_{\sigma}, f_{\gamma} \in L^2([0,T] \times \Omega; dt \otimes \bbp)$ be nonnegative progressively measurable processes. We define the coefficients $\alpha : [0,T] \times \Omega \times H \to H$, $\sigma : [0,T] \times \Omega \times H \to L_2(U,H)$ and $\gamma : [0,T] \times \Omega \times H \times Z \to H$ of the SPDE (\ref{SPDE}) as
\begin{align*}
\alpha(t,\omega,y) &:= f_{\alpha}(t,\omega) \cdot \bar{\alpha}(y),
\\ \sigma(t,\omega,y) &:= f_{\sigma}(t,\omega) \cdot \bar{\sigma}(y),
\\ \gamma(t,\omega,y,z) &:= f_{\gamma}(t,\omega) \cdot \bar{\gamma}(y,z).
\end{align*}
Then $\alpha$ and $\sigma$ are $\calb \calf \otimes \calb(H)$-measurable functions, and $\gamma$ is a $\calb \calf \otimes \calb(H) \otimes \calz$-measurable function. Moreover, let $\delta : [0,T] \times \Omega \times H \times Z \to H$ be an arbitrary $\calb \calf \otimes \calb(H) \otimes \calz$-measurable function. Then conditions (\ref{chi-growth})--(\ref{LG-gamma}) from Proposition \ref{prop-SPDE} are satisfied with $f := (L^2 \vee C^2 \vee 1) ( f_{\alpha} + f_{\sigma}^2 + f_{\gamma}^2 )$, $K := 1$ and $\chi := 0$. Consequently, for each $\calf_0$-measurable random variable $x_0 : \Omega \to H$ there exists a unique mild solution $X$ to the SPDE (\ref{SPDE}) with $X_0 = x_0$.
\end{example}


\begin{thebibliography}{20}
 
\bibitem{BLZ} Brze\'{z}niak, Z., Liu, W., Zhu, J. (2014):
  Strong solutions for SPDE with locally monotone coefficients driven by L\'{e}vy noise. \textit{Nonlinear Analysis: Real World Applications} {\bf 17}, 283--310. 
 
\bibitem{Da_Prato} Da~Prato, G., Zabczyk, J. (2014):
  \textit{Stochastic equations in infinite dimensions.} Second Edition. Cambridge University Press, Cambridge. 

\bibitem{Davies} Davies, E.~B. (1976):
  \textit{Quantum theory of open systems.} Academic Press, London.

\bibitem{SPDE} Filipovi\'c, D., Tappe, S., Teichmann, J. (2010):
  Jump-diffusions in Hilbert spaces: Existence, stability and numerics.
  \textit{Stochastics} {\bf 82}(5), 475--520.

\bibitem{Atma-book} Gawarecki, L., Mandrekar, V. (2011):
  \textit{Stochastic differential equations in infinite dimensions with applications to SPDEs.} Springer, Berlin.

\bibitem{Liu-Roeckner} Liu, W., R\"{o}ckner, M. (2015): 
\textit{Stochastic partial differential equations: An introduction.} Springer, Heidelberg.  
    
\bibitem{Prevot-Roeckner} Pr\'{e}v\^{o}t, C., R\"{o}ckner, M. (2007): 
\textit{A concise course on stochastic partial differential equations.} Springer, Berlin.

\bibitem {Salavati} Salavati, E., Zangeneh, B.~Z. (2017):
  Stochastic evolution equations with multiplicative Poisson noise and monotone nonlinearity. \textit{Bulletin of the Iranian Mathematical Society} {\bf 43}(5), 1287--1299.

\bibitem{Nagy} Sz.-Nagy, B., Foias, C., Bercovici, H., K\'{e}rchy, L. (2010): Harmonic analysis of operators on Hilbert space. Revised and Enlarged Edition. Springer, New York.  
  
\bibitem{Tappe-YW} Tappe, S. (2013):
  The Yamada-Watanabe theorem for mild solutions to stochastic partial differential equations. \textit{Electronic Communications in Probability} {\bf 18}(24), 1--13.

\bibitem{Tappe-mon} Tappe, S. (2021):
  Mild solutions to semilinear stochastic partial differential equations with locally monotone coefficients. \textit{Theory of Probability and Mathematical Statistics} {\bf 104}, 113--122.
  
\end{thebibliography}
\end{document}